\documentclass[conference]{ieeeconf}
\IEEEoverridecommandlockouts
\usepackage{cite}
\usepackage{amsmath,amssymb,amsfonts}
\usepackage{algorithmic}
\usepackage{comment}
\usepackage{graphicx}
\usepackage{textcomp}
\usepackage{xcolor}
\usepackage{dsfont}
\usepackage{amsthm}
\let\labelindent\relax
\usepackage{enumitem}
\def\BibTeX{{\rm B\kern-.05em{\sc i\kern-.025em b}\kern-.08em
    T\kern-.1667em\lower.7ex\hbox{E}\kern-.125emX}}

\newtheorem{theorem}{Theorem}[section]
\newtheorem{corollary}{Corollary}[section]
\newtheorem{lemma}[theorem]{Lemma}
\newtheorem{remark}{Remark}[section]

\newtheorem{assumption}{Assumption}

\newcommand{\sgn}{\text{sign}}
    
\begin{document}

\title{Opinion-driven risk perception and reaction in SIS epidemics\\
\thanks{This research was supported in part by ARO grants W911NF-18-1-0325 and W911NF-24-1-0126 and AFOSR grant FA9550-24-1-0002.}
\thanks{N. Leonard and M. Ordorica are with the Dept. of Mechanical and Aerospace Engineering at Princeton University, Princeton, NJ, 08544 USA; {m.ordorica,naomi}@princeton.edu}
\thanks{A. Bizyaeva is with the Sibley School of Mechanical and Aerospace Engineering at Cornell University, Ithaca, NY, 14853 USA; anastasiab@cornell.edu}
\thanks{S. Levin is with the Dept. of Ecology and Evolutionary Biology at Princeton University, Princeton, NJ, 08544 USA; slevin@princeton.edu}
}

\author{Marcela Ordorica Arango, Anastasia Bizyaeva, Simon A. Levin, Naomi Ehrich Leonard}

%\author{\IEEEauthorblockN{A1}
%\and
%\IEEEauthorblockN{A2}
%\IEEEauthorblockA{\textit{dept. name of organization (of %Aff.)} \\
%\and
%\IEEEauthorblockN{A3}

\maketitle

\begin{abstract} 

We present and analyze a mathematical model to study the feedback between behavior and epidemic spread in a population that is actively assessing and reacting to risk of infection. 
In our model, a population dynamically forms an opinion that reflects its willingness to engage in risky behavior (e.g., not wearing a mask in a crowded area)  or reduce it (e.g., social distancing). 
We consider SIS epidemic dynamics in which the contact rate within a population adapts as a function of its opinion.
For the new coupled model, we prove the existence of two distinct parameter regimes. 
One regime corresponds to a low baseline infectiousness, and the equilibria of the epidemic spread are identical to those of the standard SIS model. 
The other regime  corresponds to a high baseline infectiousness, and there is a bistability between two new endemic equilibria that reflect an initial preference towards either risk seeking behavior or risk aversion. 
We prove that risk seeking behavior increases the steady-state infection level in the population compared to the baseline SIS model, whereas risk aversion decreases it.
When a population is highly reactive to extreme opinions, we show how risk aversion enables the complete eradication of infection in the population. 
Extensions of the model to a network of subpopulations are explored numerically. 

%We present and analyze the nonlinear opinion dynamics SIS model (NOD-SIS) to understand the effects of a population's behavior toward reducing its contact rate on an epidemic's spread.
%We consider a well-mixed setting and analyze the coupling between a population's opinion to adjust its contact rate based on infection levels and the effects this opinion has on infection spread. 
%We find that the model recovers the standard SIS behavior for 
%exactly like SIS
%regime where doesnt
%We study the system's fixed points and use a Lyapunov-Schmidt reduction to study its bifurcations. 
%The model extends known results of the standard SIS and presents a transcritical bifurcation when the basic reproduction number is equal to one. 
%We prove the existence of a non-persistent unfolding of a pitchfork bifurcation where two different endemic equilibria (EE) appear. 
%In the regime where they are both stable, initial risk tolerance makes the system converge to the higher EE. 
%In contrast, risk aversion makes the system reach a lower EE, suggesting that risk aversion promotes the reduction of infection levels. 
%When opinion feedback gains are sufficiently large, numerical explorations show that risk aversion enables the complete eradication of the disease. 
%We explore numerically the extension of the well-mixed setting to a network of populations. 
%This work shows that risk aversion and risk tolerance have antagonistic effects on epidemic spread. We discuss the consequences of this in epidemic mitigation.

\end{abstract}

\section{Introduction}
Pandemics pose serious challenges to health systems. 
Analyzing how viruses spread through a population can help with the design and evaluation of control measures that reduce the impact of epidemics on human lives.
Infection spread is influenced by many factors, including the infectiousness of a disease and how quickly individuals recover from infection.
These factors are taken into account in standard compartmental epidemiological models, such as the SIS (Susceptible-Infected-Susceptible), SI, and SIR models.
These models have proved helpful in the study of disease spread, but they do not account for human behavior in response to infection nor the effects of behavior on infection spread.

Non-pharmaceutical strategies, such as the use of masks or reducing physical interactions during an epidemic, determine infection spread \cite{yang2022sociocultural, qiu2022understanding, bjornstad2020modeling}. A large body of literature has explored the interaction of a population's opinions during an epidemic and the spread of infection. In \cite{zhou2020active}, the authors present a feedback-controlled epidemic model where a population controls its contact rate as a function of infection levels, and \cite{bizyaeva2024active} extends this work by analyzing the network setting. \cite{xu2024discrete} uses a bilayer network to model the interaction between the opinions about public health concerns and infection. \cite{peng2021multilayer} employs multi-layer networks to explore how infection and opinions to engage in safe or risky behavior evolve, and \cite{she2022networked,xuan2020network,lin2021discrete} couple opinion about the severity of an epidemic and the network SIS model in continuous and discrete time. Game-theoretical approaches have also been used to explore the interplay between behavior and disease spread. \cite{paarporn2023sis} couples SIS epidemics with a replicator equation, \cite{satapathi2023coupled} explores how opinions to adopt safety measures and infection coevolve when reinfection is possible. The works of \cite{ye2021game,frieswijk2022mean} develop behavioral epidemiological models to explore how human decisions and epidemics evolve in networks in discrete and continuous time, and \cite{liu2022herd} analyzes the effects of herd behaviors in epidemics. The works \cite{doostmohammadian2023network} and \cite{doostmohammadian2020centrality} focus on the effects of network properties in epidemics and applications to mitigation and control of spread. %The papers \cite{funk2010modelling, bedson2021review} present reviews of studies that analyze the impact of behavior on epidemics. 

%We contribute to the research by examining the interconnection between opinions about adapting contact rates during an epidemic and infection spread. 
We investigate the feedback between human behavior and infection spread in a population that actively assesses the risk of infection and develops an opinion about increasing or reducing its contacts.
We introduce and analyze the nonlinear opinion dynamics SIS (NOD-SIS) model in which a population with SIS epidemic dynamics adjusts its contact rate based on its dynamic opinion about infection risk, potentially embracing one of two behavioral strategies. 
One strategy is \textit{risk seeking}, in which a population \textit{increases} its contact rates as infection levels rise. Performing essential work during a pandemic surge is an example. %of a risk tolerance strategy. 
The other strategy is \textit{risk aversion}, in which a population \textit{decreases} its contact rates as infection levels rise.
Social distancing is an example. % of a risk aversion strategy.
%A population consisting of risk-tolerant individuals will increase their contact rates as they perceive that infections are increasing. 
%A population made up of risk-averters will reduce their contacts when they perceive that infections are increasing.
When opinions about infection risk are equal to zero (i.e., neutral), the population is \textit{risk neutral}.%, and in this case, our model recovers the dynamic behavior of the standard SIS model.

This study is distinguished from previous works %that consider feedback between opinion and epidemic dynamics 
due to its consideration of a \textit{nonlinear} opinion update rule recently proposed in \cite{BizyaevaTAC}. In contrast, past works including \cite{she2022networked,xuan2020network,lin2021discrete} assume that opinions evolve through a linear averaging process. Nonlinear opinion dynamics models can make dramatically different predictions from their linear counterparts \cite{leonard2024fast}. These differences may lead to different conclusions about the effect of public opinion on the outcomes of epidemics. Our study is a rigorous examination of nonlinear effects of opinion dynamics in epidemic-behavioral models. 

Our main contributions are the following. First, we introduce the NOD-SIS model for a single population. Second, we %analyze the dynamical behavior of the nonlinear opinion dynamics SIS (NOD-SIS) model and 
examine the fixed points of the model in different parameter regimes. 
%We perform a bifurcation analysis to establish how these fixed points evolve as a function of the parameters. 
We find that for low infectiousness and basal urgency, and in a population with low peer pressure, the system behaves like the standard SIS model. For high infectiousness, two stable fixed points exist, and convergence to each one is determined by the population's initial preference towards either risk seeking or risk aversion. 
%We show that the infection levels of risk tolerators are greater than the infection levels of risk ignorers, and that risk averters have lower infection levels than risk ignorers. 
Third, we show that when peer pressure is high, the risk-averter strategy achieves a stable opinionated infection-free equilibrium. This result suggests that exercising social distancing in a population that is sensitive to peer pressure can completely eradicate infection. 
Fourth, we extend and numerically explore the NOD-SIS model in a structured population with two networks, the first representing the physical contacts between subpopulations and the second representing a social influence network with cooperative and antagonistic interactions. 
%We find that with a belief network with only cooperation, the system presents bistability where all populations choose the same strategy, and risk averters have lower infection levels than risk tolerators. If the belief network has antagonism, some populations settle at the risk averter strategy and some at the risk tolerance strategy. As in the well-mixed case, averters have lower infection levels.
%This work is structured as follows: in Section \ref{sec:background}, we present the notation and mathematical preliminaries needed for this work. We also state the standard SIS model. In Section \ref{sec:theoretical_results}, we present the theoretical results about the equilibria and bifurcations in the system in a population with small peer pressure, and we present numerical explorations about the dynamical behavior in a population with high peer pressure. In Section \ref{sec:numerical_simulations}, we explore numerically the behavior of the model in a networked population. 

%The paper is structured as follows. 
In Section \ref{sec:background} we review mathematical preliminaries and the SIS model. We define the NOD-SIS model in Section \ref{sec:NOD-SIS} and show it is well-posed. In Section \ref{sec:theoretical_results} we analyze the model in different parameter regimes. We extend to a network in Section \ref{sec:numerical_simulations} and conclude in Section \ref{sec:conclusion}.

\section{Background}
\label{sec:background}
\subsection{Mathematical Preliminaries}
$\mathbb{R}$ denotes the real numbers. For a set $X$, $|X|$ denotes its cardinality. Let $(X,\tau)$ be a topological space. For a set $\Omega\subseteq X$, the boundary of $\Omega$ is $\partial(\Omega)=\{\omega\in \Omega \mid U\cap \Omega\neq \emptyset \text{ and } U\cap(X-\Omega)\neq \emptyset\text{, for all }U\in \tau\ \text{such that }\omega\in U\}$. 
We denote by $\{x\}$ the set whose only element is $x$. 
An undirected graph $\mathcal{G}$ consists of a pair $(V,E)$ such that $V$ is a non-empty vertex set and $E\subseteq V\times V$ is an edge set of pairs of elements in $V$. 
We write $V(\mathcal{G})=V$ and $E(\mathcal{G})=E$. % to distinguish the sets of nodes and edges between different graphs. 
%Given a graph $\mathcal{G}$, 
Nodes $i,j\in V(\mathcal{G})$ are neighbors if $(i,j)\in E(\mathcal{G})$. The adjacency matrix $A_{\mathcal{G}}$ associated to  $\mathcal{G}$ is a matrix of size $|V(\mathcal{G})| \times |V(\mathcal{G})|$ such that $A_{\mathcal{G}}(i,j)=1$ if $(i,j)\in E(\mathcal{G})$ and $0$ otherwise. 
%A square matrix $A$ of size $n$ is symmetric if $A(i,j)=A(j,i)$ for all $i,j\in \{1,\ldots,n\}$. 
When $\mathcal{G}$ is  undirected, %the associated adjacency matrix 
$A_{\mathcal{G}}$ is symmetric.

The Lyapunov-Schmidt (L-S) reduction procedure, presented in  \cite{Golubitsky1985}, is a projection-based dimensionality reduction technique used in the analysis of local bifurcations in nonlinear dynamical systems. L-S reduction maps a nonlinear system to a low-dimensional representation with equilibria that are in one-to-one correspondence with those of the original system. Bifurcations of the original system are classified by analyzing the simpler low-dimensional reduced order model.
Let $F$ be a vector field $F:\mathbb{R}^n\times \mathbb{R}\to\mathbb{R}^n$, $x\in \mathbb{R}^n$ a vector of variables, and $\lambda\in \mathbb{R}$ a bifurcation parameter. Given a dynamical system $\dot{x}=F(x,\lambda)$, the fixed points of the system are given by $F(x,\lambda)=0$. Suppose that $J(x_0,\lambda_0):=D_xF(x_0,\lambda_0)$, the Jacobian of the system at $(x_0,\lambda_0)$, has a simple zero eigenvalue. The L-S reduction %is $g(x,\lambda)$, with 
$g:\mathbb{R}\times\mathbb{R}\to\mathbb{R}$ is such that the solutions of $g(x,\lambda)=0$ are in one-to-one correspondence with the fixed points of the system $\dot{x}=F(x,\lambda)$ near the singular point. Conditions for the existence of the L-S reduction are in \cite[Theorem 2.3]{10885998}.
%In particular, this procedure results in a system of algebraic equations whose zeros are in one-to-one correspondence to the fixed points of the original dynamical system near the bifurcation point.
%The L-S reduction offers a tractable method for analyzing the fixed points and bifurcations of a nonlinear high-dimensional system, such as those that appear in large interconnected network systems and would otherwise be intractable.
\subsection{SIS Model}
The %Susceptible-Infected-Susceptible (SIS) 
SIS model is a compartmental epidemiological model that describes the spread of a disease in a population when reinfection is possible. 
%The first compartmental model was presented in \cite{kermack1927contribution}. 
In the SIS model, a population is partitioned into two compartments (susceptible and infected), and agents transition between these compartments at rates that depend on the infectiousness of a disease, the contact rate between agents, and the rate at which agents recover. The proportion of infected agents in a population, denoted by $p(t)\in [0,1]$, evolves over time $t$ as
\begin{align}
    \dot{p} &= \bar\beta\alpha(1-p)  p-\delta p \label{eq:stand_sis_p},
\end{align}
where $\bar\beta > 0$ is the disease-dependent transmissibility constant, $\alpha > 0$ is the per-capita contact rate within the population, and $\delta > 0$ is the recovery rate. %The SIS model assumes that $\alpha$ is constant. In reality, agents often engage in attitudes to increase or reduce their contacts, for example, by exercising social distancing. In the following section we present the NOD-SIS model, a model to account for risk-of-infection perception and reaction and analyze their effects on infection spread.
%To account for this behavior, the actSIS model from \cite{bizyaeva2024active} replaces the constant $\alpha$ in \eqref{eq:stand_sis_p} with a saturated function that takes as argument a filtered observation of the infection levels in the population. More precisely, the actSIS equations are
%\begin{align}
%    \dot{p} &= \bar\beta\alpha(p_s)(1-p)  p-\delta p \label{eq:actsis_p},\\
%    \tau_s\dot{p}_s &= -p_s+p.
%\end{align}
%Here, $p_s$ represents the perceived levels of infection. This filtered observation serves as an argument to the function $\alpha(p_s)$, which models the adaptive contact rate of the population. The shape of $\alpha(p_s)$ determines the reaction a population has to changes in infection levels. In particular, a decreasing $\alpha(\cdot)$ represents a population where agents reduce their contact rate if they perceive infection levels rising, indicating a risk-averse attitude. An increasing $\alpha(\cdot)$ indicates tolerance to risk in a population. These two antagonistic strategies incorporate human behavior into the SIS model and examine its effects on epidemics. However, the election of the function $\alpha(\cdot)$ is determined a priori and cannot evolve as infection levels change, meaning that a population will maintain the same risk strategy throughout an epidemic. The lack of dynamism in the chosen strategy is inaccurate to observations of human behavior during a spreading process, as agents are able to change their risk strategy based on how urgent and convenient it is to engage in protective or risky behavior.

\section{NOD-SIS Model \label{sec:NOD-SIS}}
The SIS model \eqref{eq:stand_sis_p} assumes that contact rate $\alpha$ is constant within the population for the duration of the epidemic spread. 
In reality, individuals often engage in attitudes to increase or reduce their contacts, e.g., by exercising social distancing. 
We present the NOD-SIS model, which accounts for risk-of-infection perception and reaction %and analyze their effects on infection spread.
%The NOD-SIS model incorporates 
by coupling the SIS model and the nonlinear opinion dynamics (NOD) of \cite{BizyaevaTAC,leonard2024fast}. 

%In NOD-SIS, the population evolves over time its opinion 
We let $x(t) \in [-1,1]$ be the population's opinion at time $t$ of two mutually exclusive options: to decrease or increase contact rate. %, in response to its perception of risk. 
The NOD-SIS model couples the evolution of $p$ from \eqref{eq:stand_sis_p} and $x$ from \cite{BizyaevaTAC,leonard2024fast}:
%i.e. the population increases its contact rate compared to the baseline. When $x=0$, the population is indifferent to risk and has a baseline contact rate; $x<0$ represents risk aversion and an opinion to reduce contact from the baseline. $x>0$ represents risk seeking, i.e. the population increases its contact rate compared to the baseline. 
%Infection levels, denoted by $p \in [0,1]$, and a population's opinion to engage in risky or safe behavior, denoted by $x \in [-1,1]$, evolve in an interconnected way following equations
\begin{align}
    \dot{p} &= \bar\beta(1+x)(1-p)  p-\delta p,\label{eq:odsis_p_scalar}\\
    \tau_x\dot{x}&=-x+\tanh\left((k_pp + k_xx^2 + u_0)x \right) \label{eq:odsis_x_scalar}.
\end{align}
The more negative (positive) the opinion $x$, %$x<0$ ($x>0$), 
the more the population decreases (increases) contact relative to the baseline. When $x=0$ the population maintains the baseline. If the perception of risk is high, $x<0$ represents risk aversion, $x>0$ represents risk seeking, $x=0$ represents indifference to risk. 
%When $x=0$, the population is indifferent to risk and has a baseline contact rate; $x<0$ represents risk aversion, i.e. the population reduces its contact from the baseline; and $x>0$ represents risk seeking, i.e. the population increases its contact rate compared to the baseline. As in the SIS model, $\bar\beta > 0$ is the transmissibility constant of the infection and $\delta > 0$ is the recovery rate within the population. 
The parameter $\tau_x > 0$ represents the timescale of the opinion dynamics relative to the infection spread, and $u_0 \geq 0$ is the basal level of attention or urgency in the population. The constants $k_p\geq 0$ and $k_x\geq 0$ are infection and opinion feedback gains, respectively. %and $b \in \mathbb{R}$ is a bias the population holds for either of the two strategies. 
%The constant 
$k_x$ can be interpreted as the magnitude of peer pressure to modify contact as infection levels change. $k_p$ is the strength of the reaction to information about infection level. The term $u(p,x) := k_p p + k_x x^2 + u_0$ models the net urgency within the population towards forming an opinion about infection risk. Following \cite{leonard2024fast}, the quadratic term $x^2$ in the urgency is used to model how urgency changes relative to the magnitude of the opinion rather than the preference. In %the opinion dynamics 
\eqref{eq:odsis_x_scalar}, $u(p,x) = 1$ is a critical threshold: when $u(p,x) < 1$ the linear negative feedback dominates and stabilizes the neutral opinion, and when $u(p,x)>1$ the nonlinear positive feedback dominates and destabilizes the neutral opinion. The term $u(p,x)x$ is transformed by the saturating function $\text{tanh}(\cdot)$ to bound the magnitude of opinion levels. We now prove that the NOD-SIS model is well-posed.
\begin{theorem}[Positive Invariance]
\label{thm:invariance}
    Let $\Omega=[0,1]\times[-1,1]$. Then $\Omega$ is positively invariant under the flow determined by equations \eqref{eq:odsis_p_scalar} and \eqref{eq:odsis_x_scalar}.
\end{theorem}
\begin{proof}
    $\partial(\Omega)=\left(\{0\}\times [-1,1]\right)\cup\left(\{1\}\times [-1,1]\right)\cup\left([0,1]\times\{-1\}\right)\cup\left([0,1]\times\{1\}\right)$. If $(p,x)\in \{0\}\times [-1,1], \dot{p}=0$. If $(p,x)\in \{1\}\times [-1,1], \dot{p}\leq 0$. If $(p,x)\in [0,1]\times\{-1\}, \dot{x}\geq 0$, if $(p,x)\in [0,1]\times\{1\}, \dot{x}\leq 0$. By Nagumo's theorem \cite[Theorem 4.7]{blanchini2008set}, $\Omega$ is positively invariant.
\end{proof}

\section{Theoretical Results}
\label{sec:theoretical_results}

In this section, we analyze the dynamical behavior of the NOD-SIS model \eqref{eq:odsis_p_scalar},\eqref{eq:odsis_x_scalar}. We study the fixed points and bifurcations in the model and examine how risk perception and reaction affect the steady-state solutions of epidemic dynamics. First, we make a useful assumption.

\begin{assumption}
\label{as:basic_assumptions}
    i) $u_0<1$; ii) $k_p+u_0>1$. 
\end{assumption}

%The first assumption states that populations have no bias towards any strategy.
%Assumption \ref{as:basic_assumptions} carries the following interpretation.  
Assumption \ref{as:basic_assumptions}.i implies that the basal urgency towards forming an opinion is low, i.e. in the absence of peer pressure and reactivity to infection ($k_x = k_p = 0$), $u(p,x) < 1$ and resistance to forming an opinion dominates in \eqref{eq:odsis_x_scalar}. Assumption \ref{as:basic_assumptions}.ii then implies that in the absence of peer pressure ($k_x = 0$) and when the infection levels are maximal, $u(1,x) > 1$ and nonlinear effects dominate in \eqref{eq:odsis_x_scalar}. That is, the effects of peer pressure and/or reactivity to infection are necessary to modify contact rates in the population from the baseline. If the population is sufficiently reactive then it will eventually modify its behavior in response to rising infection levels even in the complete absence of peer pressure effects.

%The second assumption is consistent with a populations behavior as it states that basal urgency level is low. The second assumption states that basal urgency and reaction are, together, large enough. This also matches human behavior as it ensures reaction to infection levels.

In the following theorem we establish a transcritical bifurcation in the NOD-SIS model \eqref{eq:odsis_p_scalar},\eqref{eq:odsis_x_scalar} in which an \textit{Indifferent Infection Free Equilibrium} (IIFE) loses stability and gives rise to an \textit{Indifferent Endemic Equilibrium} (IEE).

\begin{theorem}%[Equilibria and bifurcation]
   \label{thm:transcritical} Consider \eqref{eq:odsis_p_scalar}, \eqref{eq:odsis_x_scalar}. \textbf{i)} The IIFE $(p_{IIFE},x_{IIFE})=(0,0)$ and the IEE $(p_{IEE},x_{IEE}) =(1-\frac{\delta}{\bar\beta},0)$ are equilibria for all values of $\bar\beta\in (0,1)$, $\delta\in (0,1)$, and $u_0,k_x, k_p\in (0,1)$.
   When $\bar\beta < \delta$, the IEE is outside of the trapping region $\Omega=[0,1]\times[-1,1]$ established in Theorem \ref{thm:invariance}.
   %We name these points as the Indifferent Infection-free Equilibrium (IIFE) and the Indifferent Endemic Equilibrium (IEE), respectively, as they correspond to null opinion levels. 
   \noindent \textbf{ii)} Under Assumption \ref{as:basic_assumptions}, the IIFE is locally exponentially stable for $\bar\beta<\delta$ and unstable for $\bar\beta > \delta$. %For $k_p>1-u_0$, 
   The IEE is locally exponentially stable for $\delta<\bar\beta<\bar\beta^*:=\frac{\delta k_p}{k_p-1+u_0}$ and unstable for $\bar\beta > \bar\beta^*$. 
   \noindent \textbf{iii)} Under Assumption \ref{as:basic_assumptions}, when $\bar\beta=\delta$, the NOD-SIS model undergoes a transcritical bifurcation where the IIFE exchanges stability with the IEE.
\end{theorem}
\begin{proof}
To prove i), we confirm that the points $(0,0)$ and $(1-\frac{\delta}{\bar\beta})$ are equilibria for all values of the parameters by plugging into \eqref{eq:odsis_p_scalar}, \eqref{eq:odsis_x_scalar} when $\dot{p}=0=\dot{x}$. When $\delta > \bar\beta$, $p_{IEE} < 0$ and the equilibrium is outside of the feasible trapping region $\Omega$. To prove ii), we study stability using linearization. The Jacobian of the system at $(0,0)$ is
\begin{equation}
\small
J(0,0)=\left[\begin{array}{cc}
\bar{\beta}-\delta & 0 \\
0 & \frac{1}{\tau_x}(u_0-1)
\end{array}\right].
\label{eq:jacobian_00}
\end{equation}
Thus, the IIFE is stable when $\bar\beta<\delta$ and $u_0<1$, i.e. when the eigenvalues of \eqref{eq:jacobian_00} are negative, and unstable otherwise. %We must observe that, when $\bar\beta<\delta$, the point $(1-\frac{\delta}{\bar\beta},0)$ is a fixed point of the system; however, in this case, $1-\frac{\delta}{\bar\beta}<1$. Since $p$ represents infection levels, in this case, $p$ is not in the interpretable range. 
%When $\bar\beta>\delta$, the IIFE changes its stability since one eigenvalue of the Jacobian $J(0,0)$ is positive in this parameter region. 
%On the other hand, the fixed point $(1-\frac{\delta}{\bar\beta},0)$ is stable if and only if the eigenvalues of the Jacobian 
Next, we compute the Jacobian at the IEE,
\begin{equation}
\small
J\left(1-\frac{\delta}{\bar\beta},0\right)=\left[\begin{array}{cc}
-\bar{\beta}+\delta & \delta-\frac{\delta^2}{\bar\beta} \\
0 & \frac{1}{\tau_x}((u_0-1)+k_p(1-\frac{\delta}{\bar\beta}))
\end{array}\right].
\label{eq:jacobian_iee}
\end{equation}
It follows from i) that for the IEE to be within the feasible region, we must have $\bar \beta > \delta$; then the first eigenvalue of \eqref{eq:jacobian_iee} $\lambda_1 := - \bar\beta + \delta < 0$. Since $u_0<1$ by Assumption \ref{as:basic_assumptions}.i, the first term inside the parenthesis of the eigenvalue $\lambda_2:=\frac{1}{\tau_x}\left((u_0-1)+k_p(1-\frac{\delta}{\bar\beta})\right)$ is always negative. Thus, $\lambda_2<0$ if and only if $\bar\beta<\bar\beta^*$. From Assumption \ref{as:basic_assumptions}.ii  $k_p+u_0>1$ and the region $\left[\delta,\frac{\delta k_p}{k_p+u_0-1}\right]$ is non-empty. %Recall that in the standard scalar SIS model, a transcritical bifurcation occurs at $\bar\beta = \delta$, where the IFE ($p=0$) and the EE ($p=1-\frac{\delta}{\bar\beta})$ exchange stability \cite[Lemma 3]{mei2017dynamics}. In the following theorem, we state that this transcritical bifurcation is also present in the NOD-SIS model.
%\end{proof}
%\begin{theorem}[Transcritical Bifurcation of the IIFE]
%    At $\bar\beta=\delta$, the NOD-SIS model undergoes a transcritical bifurcation where the IIFE, stable for $\bar\beta<\delta$, exchanges stability with the IEE.
%\end{theorem}
%\begin{proof}
Finally, to prove iii) we use L-S reduction. Observe $J:=J(0,0)$ has a zero eigenvalue when $\bar\beta = \delta$, and that 
%We observed that $J$ has a single zero eigenvalue. We will perform a Lyapunov-Schmidt reduction to understand the local behavior of the system close to the bifurcation point $\bar\beta=\delta$. 
%We first compute the linear projection $E:\mathbb{R}^2\to \text{range}(J)$. Simple computations give
%\begin{equation}
%E=\left[\begin{array}{cc}
%0 & 0 \\
%0 & 1
%\end{array}\right],
%\label{eq:Q}
%\end{equation}
%and thus
%\begin{equation}
%I-E=\left[\begin{array}{cc}
%1 & 0 \\
%0 & 0
%\end{array}\right].
%\label{eq:Q}
%\end{equation}
 $\text{kernel}(J)=\text{span}\{(1,0)\}$, $\text{range}(J)=\text{span}\{(0,1)\}$, and thus $\text{range}(J)^\perp=\text{span}\{(1,0)\}$. %Let $E:\mathbb{R}^2\to \text{range}(J)$ be the linear projection of $\mathbb{R}^2$ onto $\range(J)$. Observe that $E\Phi(\mathbf{y};\alpha)=\mathbf{0}$ if and only if $p=\frac{1}{k_p}\left(\frac{\text{arctanh}(x)}{x}-k_xx^2-u_0\right)$. From the implicit function theorem, we know that this expression can be inverted in a small neighborhood of $(p,x)$ to obtain an expression for $x$, that is, in a small neighborhood, we can write $x=W(p,\alpha)$. Substituting this value into $(I-E)\Phi(\mathbf{y};\alpha)=\mathbf{0}$
We compute the coefficients of the L-S reduction $g(y,\lambda)$, where $y$ is a coordinate along the linear space generated by $v=(1,0)$, the right null eigenvector of $J(0,0)$ when $\bar\beta=\delta$, and $\lambda = \bar\beta-\delta$. By performing the appropriate computations, following \cite[§3, p. 33]{Golubitsky1985} we obtain that $g_{yy}=-2\delta$, and thus $\text{sign}(g_{yy})=-1$. Also, $g_{\beta\beta}=0$, where $\beta=\bar\beta-\delta$. We compute $\text{det}\left(\begin{matrix}g_{yy} & g_{y\beta}\\ g_{y\beta} & g_{\beta\beta}\end{matrix}\right)=\text{det}\left(\begin{matrix}-2\delta & g_{y\beta}\\ g_{y\beta} & 0\end{matrix}\right) = -g_{y\beta}^2$. It only remains to prove that $g_{y\beta}\neq 0$. Straightforward computations show that $g_{y\beta}=1\neq 0$. Therefore $\text{sign(det}(d^2g))=-1$. From \cite[Proposition 9.3]{Golubitsky1985}, %we conclude that $g$ is equivalent to $\beta^2-y^2$, which is strongly equivalent to $\beta y-y^2$, the normal form of a transcritical bifurcation.
the system undergoes a transcritical bifurcation.
\end{proof}
Recall that in the SIS model~\eqref{eq:stand_sis_p}, a transcritical bifurcation occurs at $\bar\beta = \delta$, where the \textit{Infection Free Equilibrium} (IFE),
 $p=0$, and the \textit{Endemic Equilibrium} (EE), $p=1-\frac{\delta}{\bar\beta}$, exchange stability \cite[Lemma 3]{mei2017dynamics}. According to Theorem \ref{thm:transcritical}, the NOD-SIS model recovers this behavior of the SIS model. In the remainder of this section, we show that the NOD-SIS model presents richer dynamics where non-indifferent fixed points exist. We will consider two cases: weak peer pressure $k_x<\frac{1}{3}$ and strong peer pressure $k_x\geq \frac{1}{3}$. We split our analysis into these two cases because at $k_x=\frac{1}{3}$, a qualitative change occurs in the nullclines of the system. 
%These two cases present qualitatively different dynamics and correspond to low and high peer pressure in a population, respectively.
\begin{comment}
\begin{lemma}
    Consider the system determined by equations \eqref{eq:odsis_p_scalar} and \eqref{eq:odsis_x_scalar}. A point $(p^*,x^*)$ is a fixed point of the system if and only if it is the intersection of the curve $p=\frac{1}{k_p}\left(\frac{\text{arctanh}(x)}{x}-k_xx^2-u_0\right)$ with the curves $p=0$ or $p=1-\frac{\delta}{\bar\beta(1+x)}$. 
\end{lemma}
\begin{lemma}
    For $k_x<\frac{1}{3}$, the function 
\begin{equation}
f(x):=\frac{1}{k_p}\left(\frac{\operatorname{arctanh}(x)}{x}-k_xx^2-u_0\right)
\label{eq:x-nullcline}
\end{equation}
is convex. Therefore, for $k_p,k_x,u_0\in(0,1)$, $f(x)\neq 0$ for all $x\in [-1,1]$. Thus, for $k_x<\frac{1}{3}$, any fixed point $(p^*,x^*)$ of the NOD-SIS system is determined by the intersection of the curves $p=\frac{1}{k_p}\left(\frac{\text{arctanh}(x)}{x}-k_xx^2-u_0\right)$ with $p=1-\frac{\delta}{\bar\beta(1+x)}$. These two curves intersect at either zero, one, or two points. In the case where two solutions exists, let $x^*_{+}$ and $x^*_{-}$ be the larger and smaller root of 
\begin{equation}
    f_2(x):=\frac{1}{k_p}\left(\frac{\operatorname{arctanh}(x)}{x}-k_xx^2-u_0\right)+\frac{\delta}{\bar\beta(1+x)}-1.
    \label{eq:f2}
\end{equation}
Let OEE$^+=(p^*_{+},x^*_{+})$ and OEE$^-=(p^*_{-},x^*{-})$ be the opinionated endemic equilibria of the system. This name is chosen because it corresponds to positive infection levels and non-zero opinion levels.
\end{lemma}
\end{comment}
\subsection{Weak Peer Pressure}
We study the NOD-SIS model \eqref{eq:odsis_p_scalar},\eqref{eq:odsis_x_scalar} in the weak peer pressure limit $k_x<\frac{1}{3}$. 
We start by showing that for small basal urgency $u_0$, the only fixed points of the coupled system are the IEE and the IIFE, i.e. it predicts identical steady-state infection levels to those of the standard SIS model.
\begin{theorem}[SIS Equivalence] \label{thm:SIS_Equivalence}
Consider \eqref{eq:odsis_p_scalar},\eqref{eq:odsis_x_scalar}. Let $k_p, k_x, u_0\in [0,1]$ and let Assumption \ref{as:basic_assumptions} hold. For sufficiently small values of $u_0$, the IEE and the IIFE are the only fixed points of the system.
\end{theorem}
\begin{proof}
We analyze the equilibria of the system by examining its nullclines. 
%\begin{lemma}
    %Consider the system determined by equations \eqref{eq:odsis_p_scalar} and \eqref{eq:odsis_x_scalar}. 
    We observe that a point $(p^*,x^*)$ is a fixed point of the system different to the IIFE and IEE if and only if it is the intersection of the curve $p=\frac{1}{k_p}\left(\frac{\text{arctanh}(x)}{x}-k_xx^2-u_0\right)$ and the curves $p=0$ or $p=1-\frac{\delta}{\bar\beta(1+x)}$. 
%\end{lemma}
We dismiss the first of these intersections by noticing that 
%\begin{lemma}
    for $k_x<\frac{1}{3}$,
\begin{equation}
\small
f_1(x):=\frac{1}{k_p}\left(\frac{\operatorname{arctanh}(x)}{x}-k_xx^2-u_0\right)
\label{eq:x-nullcline}
\end{equation}
is convex and positive %Therefore, for $k_p,k_x,u_0\in(0,1)$, $f(x)\neq 0$
for all $x\in [-1,1]$. Thus, for $k_x<\frac{1}{3}$, any fixed point $(p^*,x^*)$ of the NOD-SIS system, different from the IIFE and IEE, is determined by the intersections of $p=\frac{1}{k_p}\left(\frac{\text{arctanh}(x)}{x}-k_xx^2-u_0\right)$ and $p=1-\frac{\delta}{\bar\beta(1+x)}$. 
Let
\begin{equation}
\small
    f_2(x):=\frac{1}{k_p}\left(\frac{\operatorname{arctanh}(x)}{x}-k_xx^2-u_0\right)+\frac{\delta}{\bar\beta(1+x)}-1,
    \label{eq:f2}
\end{equation}
%These two curves intersect at either zero, one, or two points. In the case where two solutions exist, let $x^*_{+}$ and $x^*_{-}$ be the larger and smaller root of 
%\begin{equation}
%    f_2(x):=\frac{1}{k_p}\left(\frac{\operatorname{arctanh}(x)}{x}-k_xx^2-u_0\right)+\frac{\delta}{\bar\beta(1+x)}-1.
%    \label{eq:f2}
%\end{equation}
%Let OEE$^+=(p^*_{+},x^*_{+})$ and OEE$^-=(p^*_{-},x^*{-})$ be the opinionated endemic equilibria of the system. This name is chosen because it corresponds to positive infection levels and non-zero opinion levels.
%\end{lemma}
The fixed points of the system correspond to the roots of $f_2$.
Note that $\frac{\partial f_2}{\partial u_0} < 0$ and as $u_0$ decreases, the graph of $f_2$ is translated up. We see that $f_2(x)$ is convex by computing its second derivative. We see that $\frac{\partial^2 f_2}{dx^2}\geq0$ when $\frac{\text{arctanh}(x)}{x^3}+\frac{2x^2-1}{x^2(x^2-1)^2}\geq k_x$. This follows for all $x\in [-1,1]$ when $k_x<\frac{1}{3}$. Thus, if $u_0$ is small, the only fixed points of the system are the IEE and the IIFE. %We observe this in a bifurcation diagram in the left panel of Figure \ref{fig:bifurcation-u0}. %Observe that this bifurcation diagram is qualitatively equivalent to the bifurcation diagram of the standard SIS model. 
\end{proof}
\begin{figure}
    \centering
    \includegraphics[width=\linewidth]{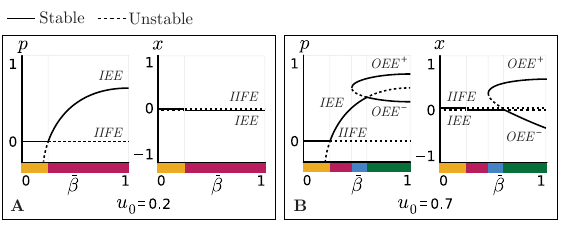}
    \caption{Bifurcation diagrams for  (A) $u_0=0.2$ and (B) $u_0=0.7$. % and $u_0=0.9$ (right). %For small basal attention levels, the system behaves exactly like the standard SIS model. As $u_0$ grows, the two opinionated endemic equilibria appear, and further, the OEE$^-$ exchanges stability with the IEE in a bifurcation that is locally like a transcritical bifurcation, and globally, the non-persistent unfolding of a pitchfork bifurcation. Parameters: $k_x=0.3, k_p=0.7,\delta=0.3$.
    For $u_0=0.7$, in the region where $\bar\beta<\delta$ (yellow), the only stable fixed point in the interpretable range is the IIFE. A transcritical bifurcation occurs when $\bar\beta=\delta$. For $\delta<\bar\beta<\bar\beta^*$ (red), the IIFE is unstable and the IEE is stable. In these two regions, the system behaves locally like the standard scalar SIS model. Let $\bar\beta_0$ be the value for which, given a set of parameters $\delta,k_p,k_x,u_0$, $f_2(x)$ has exactly one solution. For $\bar\beta \in (\bar\beta_0,\bar\beta^*)$ (blue), two new fixed points given implicitly by the roots of \eqref{eq:f2} exist, the OEE$^+$ and the OEE$^-$, the first is stable, and the latter unstable. In this region, the IIFE is unstable, and the IEE is stable. Finally, at $\bar\beta=\bar\beta^*$ the IEE exchanges stability with OEE$^{-}$ in a transcritical bifurcation. For $\bar\beta>\bar\beta^*$ (green), the only stable equilibria are the OEE$^+$ and the OEE$^{-}$. Parameters: $k_x=0.3, k_p=0.7, \delta=0.3$.}
    \label{fig:bifurcation-u0}
\end{figure}
In Theorem \ref{thm:SIS_Equivalence} we proved that small urgency results in behavior equivalent to the SIS model. This result is illustrated in the bifurcation diagrams of Fig. \ref{fig:bifurcation-u0}A.  

Next, we focus on the case where $f_2(x)$ has two real roots $x^*_{+}$ and $x^*_{-}$, where $x^*_{+} \geq x^*_{-}$. %These two solutions are associated to two fixed points OEE$^+=(p^*_{+},x^*_{+})$ and OEE$^-=(p^*_{-},x^*{-})$, where $p^*_{\pm}=1-\frac{\delta}{\bar\beta(1+x^*_{\pm})}$. 
 We show that the NOD-SIS model has richer dynamics than the standard SIS model by proving the existence of a bifurcation of the IEE for $\bar\beta>\delta$. We refer to any equilibrium of \eqref{eq:odsis_p_scalar},\eqref{eq:odsis_x_scalar} for which $x \neq 0$ and $p\neq 0$ as an \textit{Opinionated Endemic Equilibrium} (OEE). 
\begin{theorem} \label{thm:second_bif}
    Let $u_0$ be such that $f_2(x)$ in \eqref{eq:f2} has exactly two real roots $x^*_{+}$ and $x^*_{-}$, with $x^*_{+} \geq x^*_{-}$ and let OEE$^+=(p^*_{+},x^*_{+})$ and OEE$^-=(p^*_{-},x^*_{-})$, where $p^*_{\pm}=1-\frac{\delta}{\bar\beta(1+x^*_{\pm})}$.
    %We call these points Opinionated Endemic Equilibria as they correspond to positive infection and non-zero opinion levels.
    Let $\bar\beta>\delta$ and $\frac{k_p\delta}{k_p+u_0-1}<1$. Under Assumption \ref{as:basic_assumptions}, the system from \eqref{eq:odsis_p_scalar} and \eqref{eq:odsis_x_scalar} undergoes a transcritical bifurcation at $\bar\beta=\bar\beta^*$. In a small neighborhood of $(p,x,\bar\beta) = \left(1 - \frac{\delta}{\bar\beta^*} ,0,\bar\beta^* \right)$, OEE$^-$ exists for $\bar\beta < \bar\beta^*$ and is unstable, and OEE$^{-}$ exists for $\bar\beta > \bar\beta^*$ and is locally asymptotically stable.  % where the IEE becomes unstable and the OEE$^-$ becomes stable as $\bar\beta$ is increased. 
    %exhibits a non-persistent unfolding of a pitchfork bifurcation. At this point, the previously stable Indifferent Endemic Equilibrium loses stability by exchanging it with OEE$^-$. This exchange locally behaves like a transcritical bifurcation, and globally, it is the non-persistent unfolding of a pitchfork bifurcation.
\end{theorem}
\begin{proof}
    We perform a L-S reduction. Following the steps outlined in \cite[§3 p.33]{Golubitsky1985} we compute the leading coefficients of the normal form of the projection of \eqref{eq:odsis_p_scalar},\eqref{eq:odsis_x_scalar} onto the span of the right null eigenvector of $J(1-\frac{\delta}{\bar\beta},0)_{\bar\beta=\bar\beta^*}=
        \begin{bmatrix}
            \frac{\delta(u_0-1)}{k_p+u_0-1} & \frac{\delta(1-u_0)}{k_p}\\
            0 & 0
        \end{bmatrix}$
evaluated at $\beta=0$, where $\beta=\bar\beta-\bar\beta^*$, $g_y=g_{\beta}=g_{\beta \beta}=0$,  $g_{yy}=2(k_p+u_0-1)\neq 0$, and $g_{\beta y}=\frac{(k_p+u_0-1)^2}{\delta k_p}>0$. %., and $g_{yyy}=6(k_x-k_p-u_0)+4$. Observe that $g_{yyy}=0$ if and only if $k_x-k_p-u_0=-\frac{2}{3}$. In this singular case, the system undergoes a transcritical bifurcation. In the case where $k_x-k_p-u_0\neq-\frac{2}{3}$, the system exhibits a non-persistent unfolding of a pitchfork bifurcation \cite[Ch. III.7]{Golubitsky1985}.
From \cite[Proposition 9.3]{Golubitsky1985}, we establish the existence of a transcritical bifurcation and stability of the solution branches.
\end{proof}
Fig. \ref{fig:bifurcation-u0}B illustrates the secondary bifurcation whose existence was established in Theorem \ref{thm:second_bif}. Observe that when $u_0<k_p(\delta-1)+1$, the second bifurcation point $\bar\beta^*\notin[0,1]$.
Fig. \ref{fig:bifurcation-u0}B shows that the solution branches corresponding to OEE$^+$ and OEE$^-$ emerge from a single point and for all values $\bar\beta > \bar\beta^*$ there is a bistability between OEE$^+$ and OEE$^-$. The bifurcation diagram in Fig. \ref{fig:bifurcation-u0}B can be understood as a non-persistent unfolding of a pitchfork bifurcation \cite[§Ic]{Golubitsky1985}.
Note that the fixed points OEE$^+$ and OEE$^-$ correspond to risk seeking and risk aversion strategies, respectively. In the following corollary, we establish that risk seeking increases infection levels and risk aversion decreases infection levels from the baseline SIS predictions. We also prove that whether convergence is to OEE$^+$ or to OEE$^-$ is determined by the initial opinion. Let $p_{EE}^*=1-\frac{\delta}{\bar\beta}$ be the endemic infection levels of the standard SIS model \eqref{eq:stand_sis_p}. 
%For the following theorems, we will focus on the case where $u_0$ is large enough to allow the existence of the fixed points OEE$^+$ and OEE$^-$, and such that $\frac{k_p\delta}{k_p+u_0-1}<1$.
\begin{corollary}[Risk Seeking and Risk Aversion]
    Consider \eqref{eq:odsis_p_scalar},\eqref{eq:odsis_x_scalar}. Let Assumption \ref{as:basic_assumptions} hold and let $u_0$ be such that $f_2(x)$ in \eqref{eq:f2} has exactly two real roots %OEE$^+=(p_{+}^*, x_{+}^*)$ and OEE$^-=(p_{-}^*, x_{-}^*)$ exist, 
    and such that $\frac{k_p\delta}{k_p+u_0-1}<1$. Take $\bar\beta>\bar\beta^*$. Let $\Omega_S=[0,1]\times[0,1]$ and $\Omega_A=[0,1]\times[-1,0]$. The following statements hold. 
    \noindent i) There are exactly four equilibria: the IIFE, IEE, OEE$^+$, OEE$^-$; 
    \noindent ii) If $x(0) > 0 (< 0)$, then $x(t) > 0 (<0)$ for all $t > 0$. Furthermore, $\lim_{t \to \infty} (p(t),x(t)) = (p^*_{+},x^*_{+})$ 
    %or $(p^*_{-},x^*_{-})$ 
    for all initial conditions in the interior of $\Omega_S$;
    %($\Omega_A$) 
    %that are not on a periodic orbit;
    \noindent iii) $p^*_{-} \leq p_{EE}^* \leq p^*_+.$
    %Then, the opinion levels at steady state are determined completely by the sign of $x(0)$,
    %Since $f_2(x)=0$ has only two real solutions, then the only fixed points are the IIFE, the IEE, the OEE$^+$, and the OEE$^-$, and the IIFE and IEE are unstable.
    %Then, if $x(0) > 0 (< 0)$ we have that $x(t) > 0 (<0)$ for all $t > 0$, and $\lim_{t \to \infty} (p(t),x(t)) = (p^*_{+},x^*_{+})$ or $(p^*_{-},x^*_{-})$, respectively.  Furthermore, the following inequalities are satisfied:
    %\begin{equation}
    %    p^*_{-} \leq p_{EE}^* \leq p^*_+.
    %\end{equation}
\end{corollary}
\begin{proof}
    i) Since $f_2(x)=0$, this claim follows by analogous nullcline arguments as the proof of Theorem \ref{thm:SIS_Equivalence}; % has only two real solutions, then the only fixed points are the IIFE, the IEE, the OEE$^+$, and the OEE$^-$, and the IIFE and IEE are unstable.
    ii)  Observe that the set $[0,1]\times\{0\}$ is invariant under the flow of \eqref{eq:odsis_p_scalar}, \eqref{eq:odsis_x_scalar}. Recall from Theorem \ref{thm:invariance} that $ \Omega = [0,1]\times[-1,1]$ is forward invariant. Since $[0,1]\times\{0\}$ partitions $\Omega$ into $\Omega_S$ and $\Omega_A$ and no flow crosses the boundary, the two sets are themselves forward invariant. Next, observe that OEE$^+ \in \Omega_S$ and OEE$^- \in \Omega_A$ are interior points; recall that the IIFE and IEE are unstable under the parameter assumptions of this corollary following Theorems \ref{thm:transcritical} and \ref{thm:second_bif}. Observe that the off-diagonal entries of the Jacobian matrix of \eqref{eq:odsis_p_scalar},\eqref{eq:odsis_x_scalar} are $J_{12}(p,x) = \bar\beta (1-p)p$ and $J_{21}(p,x) = \frac{k_p}{\tau_x} x \operatorname{sech}^2((k_p p + k_x x^2 + u_0) x)$.  Observe that in $\Omega_S$, $J_{12}(p,x) \geq 0$ and $J_{21}(p,x) \geq 0$. This means the system is cooperative and therefore monotone in $\Omega_S$, and by \cite[Theorem 3.22]{hirsch2006monotone}, the $\omega$-limit set for any trajectory starting in the interior of $\Omega_S$ is a single equilibrium. Since OEE$^{+}$ is the only equilibrium in the set interior, all trajectories inside $\Omega_S$ approach OEE$^+$ as $t \to \infty$. %By the Poincar\'e-Bendixson theorem for planar flows \cite[Theorem 9.0.6]{wiggins2003introduction}, every $\omega$-limit set in the compact set $\Omega_S$ ($\Omega_A$) is either an isolated equilibrium or a periodic orbit. To rule out an attracting periodic orbit, we observe that the off-diagonal entries of the Jacobian matrix of \eqref{eq:odsis_p_scalar},\eqref{eq:odsis_x_scalar} are $J_{12}(p,x) = \bar\beta (1-p)p$ and $J_{21}(p,x) = \frac{k_p}{\tau_x} x \operatorname{sech}^2((k_p p + k_x x^2 + u_0) x)$.  Observe that in $\Omega_S$, $J_{12}(p,x) \geq 0$ and $J_{21}(p,x) \geq 0$ which means the system is cooperative in $\Omega_S$, and any periodic orbit in $\Omega_S$ must be unstable  \cite[Theorem 2.2]{smith1995monotone}. 
    %A trajectory starting from any point in the interior of $\Omega_S$ that is not in a periodic orbit will therefore approach OEE$^+$ as $t \to \infty$. 
    %, and therefore the sign of $x$ at steady state is determined completely by $x(0)$.
%\begin{theorem}[Aversion and Tolerance affect Infection Levels]
%     Let $u_0$ be such that OEE$^+=(p_{+}^*, x_{+}^*)$ and OEE$^-=(p_{-}^*, x_{-}^*)$ exist and $\frac{k_p\delta}{k_p+u_0-1}<1$. Let $p_{EE}^*=1-\frac{\delta}{\bar\beta}$ be the endemic infection levels of the standard SIS model. Let $\bar\beta>\frac{k_p\delta}{k_p+u_0-1}$. The following inequalities are satisfied:
%    \begin{equation}
%        p^*_{-} \leq p_{EE}^* \leq p^*_+,
%    \end{equation}
%\end{theorem}
    iii) The steady-state infection values $p_-^*$ and $p_+^*$ satisfy $p^*_{\pm}=1-\frac{\delta}{\bar\beta(1+x^*_{\pm})}$; then $ p^{*}_{\pm} - p^*_{EE} =\frac{\delta}{\bar \beta} \left(\frac{x^*_{\pm}}{1 + x^*_{\pm}}\right)$ and $
    \operatorname{sign}(p^{*}_{\pm} - p^*_{EE}) = \operatorname{sign}(x^*_{\pm})$ as long as $|x^*_{\pm}|<1$, where $x^*_+$ and $x^*_-$ are the positive and negative roots of $f_2$ in \eqref{eq:f2}. 
    %Observe that greater risk aversion leads to a greater decrease in the infection levels in the OEE$^{-}$, and greater risk seeking leads to a greater increase in the infection levels in the OEE$^+$.
\end{proof}
Since $f_2(x)$ has two real roots for large enough $u_0$, we know that the previous result is valid  in a population with high basal urgency. In Fig. \ref{fig:bifurcation-u0} we see bifurcation diagrams for the system when $u_0=0.2,0.7$. We see that the IIFE and IEE exchange stability when $\bar\beta=\delta$, and at $\bar\beta=\bar\beta^*$%\frac{\delta k_p}{k_p+u_0-1}$
, the system undergoes a second bifurcation where the IEE and the OEE$^-$ exchange stability. In this regime and for large $\bar\beta$, the system settles to a bistable endemic state where opinion and infection levels are determined completely by the sign of $x(0)$, and risk aversion results in lower infection than risk seeking.
Fig.~\ref{fig:trajectories} shows trajectories for random initial conditions and $u_0=0.7$ and for different values of $\bar\beta$ in $[0,1]$. These trajectories show that $\bar\beta<\bar\beta^*$ leads to behavior of the SIS model, while for $\bar\beta>\bar\beta^*$, the richer dynamics of the NOD-SIS model distinguish the risk seeking and risk aversion strategies. 
\begin{figure}
    \centering
    \includegraphics[width=\linewidth]{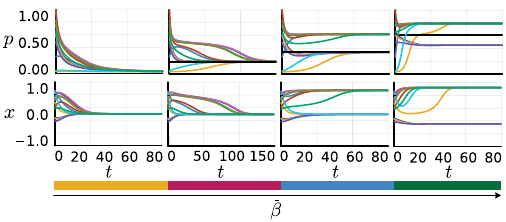}
    \caption{Trajectories for $12$ random initial conditions for $\bar\beta = 0.25$, $\bar\beta=0.36$, $\bar\beta =0.44$, and $\bar\beta =0.75$, that correspond to each of the regions (yellow, red, blue and green) of Fig. \ref{fig:bifurcation-u0}B when $u_0=0.7$. The black line in the infection plots is the endemic equilibrium of the standard SIS model. For $\bar\beta=0.25,0.36$, the system converges to the IIFE and IEE, respectively. For $\bar\beta=0.44$, agents who begin with an averter strategy converge to the endemic equilibrium of the SIS model, while agents who start with a risk seeking strategy converge to a higher infection level. For $\bar\beta=0.75$, the system's trajectories converge to one of the two opinionated equilibria determined by the initial opinions. Parameters: $\delta=0.3,k_p=0.7,k_x=0.3,\tau_x=1$.}
    \label{fig:trajectories}
\end{figure}
In the next section, we explore numerically the case when %the opinion feedback gain 
$k_x$ is high, and the function $f_1(x)$ defined in \eqref{eq:x-nullcline} is not convex. We see that for certain parameter regimes, risk aversion allows the complete eradication of infection, while risk seeking increases infection levels.

\subsection{Strong Peer Pressure}
In this section we explore numerically the behavior of the system when peer pressure $k_x$ is large. We see that a stable \textit{Opinionated Infection Free Equilibrium} (OIFE) exists with the risk averter strategy, and a symmetric stable OIFE point does not exist with the risk seeking or risk-neutral strategy. 
%We begin with the following remark:
\begin{remark}
For $k_x>\frac{1}{3}$, the function $f_1(x)$ in \eqref{eq:x-nullcline} is not convex, and we can find $u_0$, $k_p$ and $k_x$ such that $f_1(x)=0$.
%\begin{equation}
%\frac{1}{k_p}\left(\frac{\operatorname{arctanh}(x)}{x}-k_xx^2-u_0\right)=0.
%\label{eq:new_roots}
%\end{equation}
%\begin{figure}
%    \centering
%    \includegraphics[width=\linewidth]{Figures/kxlarge.png}
%    \caption{Caption}
%    \label{fig:enter-label}
%\end{figure}
%More new fixed points. The shape of the x nullcline now changes significantly, and there are more fixed points. I will make a separate analysis of this case. Alternatively, take the whole urgency as a "bifurcation parameter." When it is low, we just have the standard SIS. Then, as it grows, we have what is in the last section, and when it grows more, we get a different behavior, and this section will explore that behavior. I think it would be cool to talk about the urgency if, for example, we are dealing with a specific disease for which $\delta$ and $\bar\beta$ are fixed, and $\bar\beta>\delta$, so the IFE is no longer stable. By increasing the urgency "artificially" significantly, in particular, increasing $k_x$, new non-indifferent infection-free equilibrium points appear (I think four appear and two are stable). Meaning that the system can converge to an infection-free state even when the disease is highly infectious, i.e., $\bar\beta$ is high, and I think this is cool because it means that with high urgency, it is possible to go to no disease, at least locally and for some small initial conditions.
The solutions to this equation %\eqref{eq:new_roots} 
correspond to null infection and non-zero opinion levels, and do not depend on the values of $\delta$ and $\bar\beta$. Thus, even for a large value of $\bar\beta$, associated with very infectious diseases, an OIFE exists. 
\end{remark}
In numerical simulations we see that only one of the OIFE associated to the roots of $f_1(x)$ %\eqref{eq:new_roots} 
is stable, and it corresponds to negative opinion levels (risk aversion). We see that the steady-state behavior is determined by $\sgn(x(0))$. When $x(0)<0$, i.e., when the initial opinion is towards risk aversion, the population reaches the stable OIFE and thus eliminates the disease. If $x(0)>0$, the population reaches an OEE, and the infection levels are higher than the EE in the SIS. This demonstrates an absence of symmetry in infection levels associated with the different strategies. It suggests that in a population with high sensitivity to opinion levels and urgency levels, risk aversion is beneficial as it leads to an infection-free state. We leave the analysis of the equilibria and bifurcations in this parameter regime for future work.
\begin{figure}
    \centering
    \includegraphics[width=\linewidth]{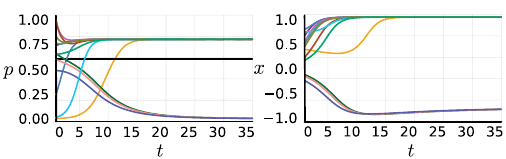}
    \caption{Initial opinion towards the risk seeking or aversion strategy is reinforced in a population with high peer pressure, and the sign of initial opinions are determinant of the infection levels at steady state. Initial averters reach an infection-free state, while initial risk seekers reach an endemic state with higher infection levels than the endemic equilibrium of the standard SIS infection level, represented by a thick black line. Parameters: $\delta = 0.3, \bar\beta = 0.75, u_0 = 0.9, k_p = 0.7, k_x = 0.7.$}
    \label{fig:largekx}
\end{figure}

\section{Numerical Simulations for Structured Populations} 
\label{sec:numerical_simulations}
We explore the behavior of the NOD-SIS model in  a structured population. 
Variables $p_j$ and $x_j$ are population $j$ infection and opinion levels. Parameter $\delta_j$ is the recovery rate in population $j$. 
We consider two networks with graph adjacency matrices $A$ and $\hat{A}$. $A$ represents the physical contacts between subpopulations: %, i.e., it encodes how subpopulations physically interact with each other. 
edge $(i,j)\in E(A)$ if and only if subpopulation $i$ has physical contact with subpopulation $j$.  $\hat{A}$ encodes communication in, for example, a social online network: edge $(i,j)\in E(\hat{A})$ if and only if subpopulation $i$ shares information with subpopulation $j$. We assume $A$ and $\hat{A}$ are symmetric, connected, and that $a_{ii}=1$ and $\hat{a}_{ii}=1$ to account for transmission within subpopulations. We use two distinct networks to distinguish virus transmission from information transmission.
%, governed by $A$, and opinion transmission, governed by $\hat{A}$. 
The dynamics are %in a structured population are
\begin{align}
    \dot{p}_j &= \bar\beta(1+x_j)(1-p_j)  \sum_{k=1}^N a_{jk}p_k-\delta_j p_j,\label{eq:odsis_p_network}\\
    \tau_x\dot{x}_j&=-x_j+\tanh\left(u_j\cdot\left(\sum_{\substack{k=1}}^N \hat{a}_{jk}x_k\right)\right) \label{eq:odsis_x_network},
    %u_j &= k_p\frac{1}{\hat{d}_j}\sum_{k=1}^N \left|\hat{a}_{jk}\right| p_k + k_x\sum_{k=1}^N \hat{a}_{jk}x_j^2 + u_0.\label{eq:odsis_u_network}
\end{align}
\normalsize
where $u_j:= k_p\frac{1}{\hat{d}_j}\sum_{k=1}^N \left|\hat{a}_{jk}\right| p_k + k_x\sum_{k=1}^N \hat{a}_{jk}x_j^2 + u_0$ and $\hat{d}_j=\sum_{k=1}^N a_{jk}$. These equations generalize \eqref{eq:odsis_p_scalar} and \eqref{eq:odsis_x_scalar} accounting for the role of the networks $A$ and $\hat{A}$. We assume $\delta_j=\delta$, i.e. all subpopulations recover at the same rate $\delta$. %For our explorations, 

We compare this system with the standard network SIS model~\cite{mei2017dynamics} in Fig. \ref{fig:ant}. Column 1 shows contact network $A$ and two different communication networks: $\hat{A}_{coop}$ has all positive edges corresponding to cooperation between subpopulations, and $\hat{A}_{ant}$ has negative edges corresponding to antagonism. %$A$ is a wheel graph with $5$ nodes, and in both cases $\hat{A}$ is a complete graph of the same size. %We first explore the case where all the edges of $\hat{A}$ are positive, as it describes a cooperation between subpopulations. 
%For these cases, we explore how trajectories behave in comparison to the standard network SIS model for different sets of initial conditions. The second and third columns  of Fig. \ref{fig:ant} correspond to trajectories using $A$ and $\hat{A}_{coop}$. 
In columns 2 and 3, cooperation makes all subpopulations reach either a risk seeking (red) or risk aversion (green) strategy, and the common choice determines the steady-state infection level. As in the well-mixed case, common risk aversion results in lower infection levels for all subpopulations as compared to the standard network SIS, while risk seeking behavior increases infection levels. 
In column 4 of Fig. \ref{fig:ant} %we explore how antagonism affects trajectories using $\hat{A}_{ant}$ as information network. In this case 
antagonism leads to different subpopulations settling at different strategies. Risk aversion subpopulations reach lower infection levels than risk seeking subpopulations. %Future work will explore how the topology of the contact and communication networks determines infection levels.

\begin{comment}
\begin{figure}
    \centering
    \includegraphics[width=\linewidth]{Figures/coopc.pdf}
    \caption{Contact and communication graphs $A$ and $\hat{A}_{coop}$ (left) for $5$ subpopulations where cooperation between populations is present. We see that populations reach a state of agreement for either risk aversion (middle) or risk seeking (right). The strategy chosen determines infection levels, and aversion reduces infection levels with respect to the standard network SIS model for the same initial conditions, while risk seeking increases infection levels. Parameters: $\bar\beta = 0.5,\delta = 0.3, k_p = 0.5,k_x = 0.3, u_0 = 0.7$.}
    \label{fig:networks_coop}
\end{figure}
\end{comment}

\begin{figure}
    \centering
    \includegraphics[width=\linewidth]{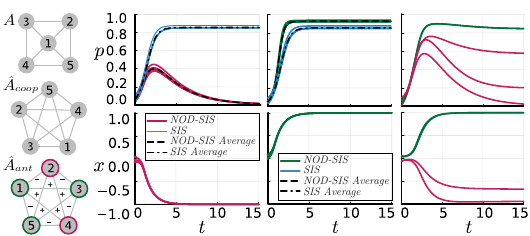}
    \caption{Column 1 shows contact graph $A$ and communication graphs $\hat{A}_{coop}$ and $\hat{A}_{ant}$ for $5$ subpopulations. Columns 2 and 3 show that under a cooperation regime ($\hat{A}=\hat{A}_{coop}$), subpopulations reach a state of agreement for either risk aversion (col. 2) or risk seeking (col. 3). The strategy chosen determines infection levels, and aversion reduces infection levels with respect to the standard network SIS model for the same initial conditions, while risk seeking increases infection levels.  In column 4 we see that antagonism between subpopulations ($\hat{A}=\hat{A}_{ant}$) results in disagreement and some subpopulations choose risk aversion (red) and others choose a risk seeking strategy (green). %Risk aversion reduces infection while risk seeking increases infection levels. 
    For all graphs $a_{ii}=1$ and $\hat{a}_{ii}=1$ but not shown in graphs. Parameters: $\bar\beta = 0.5,\delta = 0.3, k_p = 0.5,k_x = 0.3, u_0 = 0.7$.}
    \label{fig:ant}
\end{figure}

\section{Conclusion and Future Directions}
\label{sec:conclusion}
We presented the NOD-SIS model to couple the epidemiological SIS model with opinion dynamics in a well-mixed population. %We analyzed the case when there is 
For low peer pressure and low infectiousness, the system behaves locally like the standard SIS model. For higher infectiousness, the system presents a state of bistability where a population's initial opinion for risk seeking or risk aversion increases or decreases the steady-state infection levels when compared to the basal SIS model. % and the sign of the initial opinion determines the basin of attraction of each state. We showed that risk aversion reduces infection levels, and risk seeking increases infection with respect to the standard SIS model. 
For high peer pressure and high basal urgency, initial risk aversion drives the system to an \textit{opinionated  infection-free} equilibrium. %that can be reached with initial risk aversion. %In this regime, risk seeking drives the system to positive infection levels. 

We explored the NOD-SIS model in a structured population using two networks among subpopulations: a contact network to model infection spread and a communication network to model information spread. When the communication network is cooperative, all subpopulations choose risk aversion or all choose risk seeking. %As in the well-mixed setting, risk aversion (seeking) decreases (increases) infection levels relative to the infection levels of the network SIS. % for the same initial conditions. 
When the communication network has antagonistic interactions, %the system reaches a state of disagreement where 
some subpopulations choose risk aversion and some choose risk seeking. %Averters reach lower infection levels than risk seekers. 
In future work, we will analyze the dynamical properties in the network setting and explore how the two different networks influence the system's outcomes. We will also test the NOD-SIS model in real data to explore its policy implications.

\begin{comment}
\section*{Appendix}
We compute the L-S Reduction for Theorem \ref{thm:transcritical}. Let $v=(1,0)^T$, $w=(1,0)^T$, and $\beta=\bar\beta-\delta$. We proceed to compute the L-S coefficients following \cite[Equation 3.23]{Golubitsky1985}. We define $h_1((p,x),\bar\beta):=\bar\beta(1+x)(1-p)p-\delta p$ and $h_2((p,x),\bar\beta):=\frac{1}{\tau_x}(-x+\tanh((k_xx^2+k_pp+u_0)x)$, and let $F((p,x),\bar\beta):=\left(h_1((p,x),\bar\beta), h_2((p,x),\bar\beta) \right)^T$. Then  $g_{\beta}:=\left\langle w,F_{\beta}((0,0),0)\right\rangle=0$. Next, let $y=(p,x)$, then $g_{yy}:=\left\langle w, d^2F_{((0,0),0)}(v,v)\right \rangle=-2\delta$, and since $\text{sign}(\delta)=1$, then $\text{sign}(g_{yy})=-1$. Finally, $g_{\beta y}=\langle w, dF_{\beta}\cdot v\rangle=1\neq0$. Finally, $\text{sign} \begin{vmatrix}
g_{yy} & g_{\beta y}\\
g_{\beta y} & g_{\beta \beta}
\end{vmatrix}=\begin{vmatrix}
-2\delta & 1\\
1 & 0
\end{vmatrix}=-1$.
\end{comment}

\bibliographystyle{ieeetr}
\bibliography{bib_av}

\end{document}